\documentclass[11pt,reqno]{amsart}

\usepackage{amsmath}
\usepackage{amssymb}
\usepackage{amscd}
\usepackage{amsfonts}

\newtheorem{theorem}{Theorem}[section]
\newtheorem{lemma}[theorem]{Lemma}
\newtheorem{proposition}[theorem]{Proposition}

\newtheorem{remark}[theorem]{Remark}
\newtheorem{corollary}[theorem]{Corollary}

\numberwithin{equation}{section}

\begin{document}
\title[]{ Aronson-B\'enilan estimates for the fast diffusion equation under the Ricci flow}

\author{Huai-Dong Cao and Meng Zhu}

\address{Huai-Dong Cao \\ Department of Mathematics\\Lehigh University \\Bethlehem, PA 18015, USA} \email{huc2@lehigh.edu}

\address{Meng Zhu\\ Department of Mathematics\\East China Normal University\\ Shanghai 200241, China \& Department of Mathematics\\University of California at Riverside\\Riverside, CA 92521, USA}
\email{mzhu@math.ecnu.edu.cn, mengzhu@ucr.edu}
\date{}
\thanks{The research of the first author was partially supported by NSF Grant DMS-0909581.}
\thanks{The research of the second author was partially supported by the National Natural Science Foundation of China Grant No.11501206.}
\maketitle

\begin{abstract}
We study the fast diffusion equation with a linear forcing term,
\begin{equation}\label{eq0}
\frac{\partial u}{\partial t}=\textrm{div}(|u|^{p-1}\nabla u)+Ru,
\end{equation}
under the Ricci flow on complete manifold $M$ with bounded curvature and nonnegative curvature operator, where $0<p<1$ and $R=R(x,t)$ is the evolving scalar curvature of $M$ at time $t$. We prove Aronson-B\'enilan and Li-Yau-Hamilton type differential Harnack estimates for positive solutions of \eqref{eq0}. In addition, we use similar method to prove certain Li-Yau-Hamilton estimates for the heat equation and conjugate heat equation which extend those obtained by X. Cao and R. Hamilton \cite{CaHa2009}, X. Cao \cite{xCao2008}, and S. Kuang and Q. Zhang \cite{KuZh2008} to noncompact setting.
\end{abstract}

\section{Introduction}

This is a sequel to our paper \cite{CaZh2014}, in which we derived Aronson-B\'enilan and Li-Yau-Hamilton type differential Harnack estimates for positive solutions of the porous medium equation
coupled with the Ricci flow on complete manifolds with bounded curvature and nonnegative curvature operator. In this paper, we study similar problems for the fast diffusion equation under the Ricci flow.

On a Riemannian manifold $(M^n, g_{ij})$, the fast diffusion equation (FDE) is the nonlinear parabolic equation
\begin{equation}\label{eq1}
\frac{\partial u}{\partial t}=\textrm{div}(|u|^{p-1}\nabla u),
\end{equation}
where $p<1$. Equation \eqref{eq1} is also known as the porous medium equation (PME) when $p>1$. PME and FDE appear naturally as nonlinear diffusion models in many areas of mathematics, physics, chemistry, and engineering. For example, the PME is the model of the flow of gas through porous medium, ground water filtration, and heat radiation in plasmas, etc (see e.g. \cite{Vaz2007}), while FDE arises in plasma physics (\cite{OkDa1973}), diffusion of impurities in silicon (\cite{King1988}), gas kinetics theory (\cite{LiTo1997}), etc. Moreover, in geometry, the FDE with $n=2$ and $p=0$ describes the Ricci flow on surfaces (\cite{Wu1993}, \cite{DHS2012}), and it corresponds to the Yamabe flow when $n\geq3$ and $p=\frac{n-2}{n+2}$ (\cite{Vaz2006}).

While both PME and FDE look quite similar to the heat equation, the case when $p=1$ in \eqref{eq1}, the nonlinearity and degeneracy make their existence and regularity theory differ a lot from the heat equation. On Euclidean spaces, a relatively satisfying theory for the PME has been developed over the past several decades. Readers may refer to \cite{Vaz2007} and the references therein for a thorough account of the PME. Regarding the FDE on $\mathbb{R}^n$, see e.g. \cite{ArBe1979}, \cite{HePi1985}, \cite{Dake1988}, \cite{DaPi1994}, \cite{Dapi1995}, \cite{Dapi1997}, \cite{Dapi1999}, \cite{Ber2001}, and \cite{Vaz2006}. While many properties and techniques in the PME case can be shared in the FDE case, the two equations have some fundamental differences. For instance, long time solution to the Cauchy problem of FDE with $0<p<1$ and $L^1_{loc}$ initial condition exists without any growth restriction on the initial value (\cite{HePi1985}), solutions of FDE with $p<0$ and $L^1$ initial data (see e.g. \cite{Vaz1992}) do not exist, and the $L^1$ to $L^{\infty}$ smoothing effect can only be carried from PME to FDE for $\frac{n-2}{n}<p<1$ (see e.g. \cite{Vaz2006}).

If one concentrates on nonnegative solutions of \eqref{eq1} with $p>0$, it is equivalent to the following equation:
\begin{equation}\label{eq FDE}
\frac{\partial u}{\partial t}=\Delta u^p.
\end{equation}

In the study of Cauchy problems of PME and FDE on $\mathbb{R}^n$, an essential tool is the following second order differential inequality discovered by Aronson and B\'enilan (\cite{ArBe1979}):
\begin{equation}
\sum_i\frac{\partial}{\partial x_i}(pu^{p-2}\frac{\partial u}{\partial x_i})\geq-\frac{\kappa}{t},
\end{equation}
where $\kappa=\frac{n}{2+n(p-1)}$, for any positive solution $u$ to \eqref{eq FDE} with $p>(1-\frac{2}{n})^+$.

On a general complete Riemannian manifold $(M^n, g_{ij})$ with Ricci curvature bounded from below, $Rc\geq -K$ for some $K\geq0$, such kind of differential estimate was first found by Li and Yau (\cite{LiYa1986}). Subsequently, the matrix differential Harnack estimate for the heat equation was proved by Hamilton (\cite{Ham1993cag}) on Riemannian manifolds, and by Ni and the first author (\cite{CaNi2005}) on K\"ahler manifolds.

Unlike the heat equation, PME and FDE on Riemannian manifolds were only investigated recently. Aronson-B\'enilan type estimates were first proved by V\'azquez (\cite{Vaz2007}) on manifolds with nonnegative Ricci curvature, and generalized by Lu, Ni, V\'azquez and Villani (\cite{LNVV2009}) to manifolds with Ricci curvature bounded from below. In particular, they showed that, if $Rc\geq-(n-1)K$ on $M$ for some $K\geq0$, then for any $\alpha\in(0,1)$, $\epsilon>0$ and bounded smooth positive solution $u$ to FDE, one has
\begin{equation}\label{eqNi}
\alpha\frac{v_t}{v}-\frac{|\nabla v|^2}{v}\leq \frac{\kappa\alpha^2}{C}\left[\frac{1}{t}+\frac{(1-p)(n-1)}{(1-\alpha)\epsilon}\sqrt{C}K\bar{v}_{max} \right],
\end{equation}
where $v=\frac{p}{p-1}u^{p-1}$, $\kappa=\frac{n(1-p)}{2-n(1-p)}$, $\bar{v}_{max}=\max_{M\times[0,T]}v$, and $C=C(\kappa, \alpha,\epsilon)>0$.

The study of differential Harnack estimates for parabolic equations under the Ricci flow
\begin{equation}\label{rf}
\frac{\partial g_{ij}}{\partial t}=-2R_{ij}
\end{equation}
on complete manifolds was originated from Hamilton's work (\cite{Ham1988}). The readers may refer to e.g. \cite{Cao1992}, \cite{Cho1991}, \cite{Ham1993jdg}, \cite{Ni2006}, \cite{Ni2007}, \cite{Per2002a} for more results on differential Harnack estimates for linear and semi-linear parabolic equations under the Ricci flow and the K\"ahler-Ricci flow.

In recent years, among the parabolic equations coupled with the Ricci flow, the heat equation
\begin{equation}\label{heat equation1}
\frac{\partial u}{\partial t}=\Delta u,
\end{equation}
and the conjugate heat equation
\begin{equation}\label{conjugate heat equation1}
\frac{\partial u}{\partial t}=-\Delta u+ Ru
\end{equation}
are two of the most studied.

In \cite{Per2002a}, Perelman first obtained a Li-Yau type estimate for the fundamental solution of the conjugate heat equation \eqref{conjugate heat equation1} (see also \cite{Ni2006}). On compact manifolds, X. Cao \cite{xCao2008} and Kuang-Zhang \cite{KuZh2008} independently showed the following Li-Yau estimate,
\begin{equation}\label{kuang-zhang}
2\Delta v-|\nabla v|^2+R\leq \frac{n}{2\tau},
\end{equation}
where $u=(4\pi\tau)^{-n/2}e^{-v}$, for arbitrary positive solutions $u$ of \eqref{conjugate heat equation1}.

Regarding the heat equation \eqref{heat equation1} under the Ricci flow, Li-Yau type estimates were first studied by Bailesteanu-X.Cao-Pulemotov \cite{BCP} assuming the Ricci curvature being uniformly bounded. Most recently, Q. Zhang and the second author \cite{qZZ}  were able to show a Li-Yau type estimate only under the assumption that the scalar curvature is uniformly bounded.

Moreover, By using Hamilton's trace differential Harnack estimate \cite{Ham1993jdg}, X. Cao and Hamilton \cite{CaHa2009} proved the Li-Yau-Hamilton type estimate
\begin{equation}\label{Cao-Hamilton}
2\Delta v-|\nabla v|^2-3R\leq \frac{2n}{t},
\end{equation}
where $v=-\ln u$ and $u$ is any positive solution to the heat type equation
\begin{equation}\label{conjugate equation}\frac{\partial u}{\partial t}=\Delta u+Ru\end{equation}
under the Ricci flow on a compact manifold with nonnegative curvature operator.

In \cite{CaZh2014}, we studied PME \eqref{eq FDE} with a linear forcing term,
\begin{equation}\label{eq2}
\frac{\partial u}{\partial t}=\Delta u^{p}+ Ru,
\end{equation}
where $p>1$, under the Ricci flow on a complete manifold $M^n$ with bounded curvature and nonnegative curvature operator and proved Aronson-B\'enilan and Li-Yau-Hamilton type differential Harnack estimates for any bounded positive solution $u(x,t)$. Especially, we showed
$$\frac{|\nabla v|^2}{v}-2\frac{v_t}{v}-\frac{R}{v}-\frac{d}{t}\leq 0,$$
where $v=\frac{p}{p-1}u^{p-1}$, $d=\max\{2\alpha, 1\}$ and $\alpha=\frac{n(p-1)}{1+n(p-1)}$. The reason for adding the term $Ru$ in \eqref{eq2} is that the total mass $\int_M u dV$ would stay unchanged under \eqref{eq2} coupled with the Ricci flow, as it does under \eqref{eq FDE} in the fixed metric case.

In this paper, we consider the equation \eqref{eq2} with $0<p<1$, i.e., the FDE case. We prove that certain Aronson-B\'enilan and Li-Yau-Hamilton type estimates as in \cite{CaZh2014} still hold. In particular, we have

\begin{theorem}
Let $(M^n, g_{ij}(t))$, $t\in[0, T]$, be a complete solution to the Ricci flow with bounded curvature and nonnegative curvature operator. If $u$ is a bounded smooth positive solution to \eqref{eq2} with $p\in(1-\frac{1}{n},1)$, then for $v=\frac{p}{p-1}u^{p-1}$, we have
$$\frac{|\nabla v|^2}{v}-2\frac{v_t}{v}-\frac{R}{v}+ KR_{max}+\frac{2\alpha}{t}\geq 0$$
on $M\times(0,T]$, where $\alpha=\frac{n(1-p)}{1-n(1-p)}$, $R_{max}=\sup_{M\times[0,T]}R$, and $K=(1-p) \max\{(\alpha+\frac{1}{2}),\,  2\alpha\}.$
\end{theorem}

A more general version of the differential Harnack inequality depending on a positive parameter $b>1$ (see Theorem \ref{refined noncompact harnack p<1 b<2} and Theorem \ref{refined noncompact harnack p<1 b>2}) is given in Section 2, and Theorem 1.1 is the special case when $b=2$. It happens that the proof of these differential Harnack inequalities in the FDE case is more subtle than in the PME case, in the sense that
some terms on the right hand side of Eq. \eqref{eq F} in Proposition 2.1 become bad terms in the FDE case, while they are good terms in the PME case. Thus, more effort was needed in order to have them controlled.

\begin{remark}
When the metric is fixed, Lu-Ni-V\'azquez-Villani \cite{LNVV2009} obtained a series of Aronson-B\'enilan estimates for FDE \eqref{eq1} in which the parameters $b\in(0,1)$ and $p\in(1-\frac{2}{n},1)$ (see \eqref{eqNi}). However, under the Ricci flow, an extra term (the second term on the right hand side of \eqref{eq F} in Proposition 2.1) appears due to the fact that the metrics are evolving. In order to make the maximum principle argument work, it is required that $b>1$.
Moreover, our method used in the proofs of Theorem \ref{refined noncompact harnack p<1 b<2} and Theorem \ref{refined noncompact harnack p<1 b>2} can also be applied to the FDE \eqref{eq1} in the fixed metric case to prove certain Aronson-B\'enilan estimates with $b>1$ and $p\in(1-\frac{2(b-1)}{bn},1)$.
\end{remark}

Furthermore,  we apply similar techniques used in the proof of Theorem 1.1 to extend the Li-Yau-Hamilton estimates \eqref{kuang-zhang} and \eqref{Cao-Hamilton}, for the conjugate heat equation and the heat equation under the Ricci flow respectively, to the complete noncompact setting.

\begin{theorem}\label{theorem conjugate heat equation}
Let $g(t)$, $t\in[0, T)$, be a complete solution to the Ricci flow on $M^n$ with bounded curvature and nonnegative scalar curvature. If $u$ is a positive solution to the conjugate heat equation \eqref{conjugate heat equation1}, then for $v=-\ln u-\frac{n}{2}\ln(4\pi(T-t))$, we have
$$2\Delta v-|\nabla v|^2+R\leq \frac{2n}{T-t}$$
on $M\times[0,T)$.

\end{theorem}

\begin{theorem}\label{theorem heat equation}
Let $(M^n, g_{ij}(t))$, $t\in[0, T)$, be a complete solution to the Ricci flow with bounded curvature and nonnegative curvature operator. If $u$ is a positive solution to the heat equation \eqref{conjugate equation} and $v=-\ln u$, then we have
\begin{align}\label{eq noncompact harnack2 p=1}
2\Delta v-|\nabla v|^2-3R\leq \frac{2n}{t}
\end{align}
on $M\times(0, T)$.
\end{theorem}



The proofs of Theorem \ref{theorem conjugate heat equation} and Theorem \ref{theorem heat equation} are given in Section 3 and Section 4, respectively.\\

\hspace{-.4cm}\textbf{Ackownledgements}:  Part of the work was carried out while the first author was visiting the University of Macau, where he was partially supported by Science and Technology Development Fund (Macao S.A.R.) Grant FDCT/016/2013/A1, as well as the Projects RDG010 and MYRG2015-00235-FST of University of Macau . The second author would like to thank Professors Qing Ding, Jixiang Fu, Jiaxing Hong, Jun Li, Quanshui Wu, Qi S. Zhang and Weiping Zhang for their support and encouragement. The second author also wants to thank the Shanghai Center for Mathematical Sciences, where part of the work was carried out, for its hospitality.
\bigskip

\section{Aronson-B\'enilan and Li-Yau-Hamilton estimates for FDE}

In this section we study the FDE \eqref{eq2} (with $0<p<1$) coupled with the Ricci flow \eqref{rf},  and prove Aronson-B\'enilen and Li-Yau-Hamilton type estimates for its positive solutions.

Suppose $u=u(x,t)$ is a positive solution to the FDE \eqref{eq2}, with $0<p<1$. Notice that the function $v=\frac{p}{p-1}u^{p-1}$ satisfies the equation
\begin{equation}\label{eqv}
\frac{\partial v}{\partial t}=(p-1)v\Delta v + |\nabla v|^2 + (p-1)Rv.
\end{equation}

Let
$$F=\frac{|\nabla v|^2}{v}-b\frac{v_t}{v}+(1-b)\frac{R}{v}.$$
If we denote by $$y=\frac{|\nabla v|^2}{v}+\frac{R}{v}\qquad \textrm{and} \qquad z=\frac{v_t}{v}+\frac{R}{v},$$
then we have
$$F=y-bz.$$

Next we recall the following result from \cite{CaZh2014} (see Proposition 2.2 in \cite{CaZh2014}):
\begin{proposition}\label{evolution F} Suppose $u$ is a smooth positive solution to \eqref{eq2} and $v=\frac{p}{p-1}u^{p-1}$. Let $$\mathcal{L}=\frac{\partial}{\partial t}-(p-1)v\Delta.$$ Then,
\begin{equation}\label{eq F}
\begin{aligned}
\mathcal{L}(F)&=2p\nabla_i F\nabla_iv-[\frac{b-1}{v}+p-1](\frac{\partial R}{\partial t}-2\nabla_i R\nabla_i v+2R_{ij}\nabla_i v \nabla_j v)\\
&\quad\ -2(p-1)|\nabla^2 v+\frac{b}{2}Rc|^2+\frac{(b-2)^2}{2}(p-1)|Rc|^2-\frac{1}{b}F^2\\
&\quad\ -[(p-1)R+\frac{2(b-1)}{b}\frac{R}{v}]F -\frac{b-1}{b}y^2-\frac{(b-1)(b-2)}{b} y \frac{R}{v}.
\end{aligned}
\end{equation}
\end{proposition}

We also need the following trace version of Hamilton's matrix Li-Yau estimate for the Ricci flow.
\begin{theorem}[\textbf{Hamilton \cite{Ham1993jdg}}]\label{matrix harnack}
Let $(M^n, g_{ij}(t))$, $t\in[0, T)$, be a complete solution to the Ricci flow with bounded curvature and nonnegative curvature operator, then for any 1-form $V_i$ on $M^n$, we have
$$\frac{\partial R}{\partial t} +2\nabla_i R V_i + 2R_{ij}V_iV_j + \frac{R}{t}\geq0$$
for $t\in(0,T)$.
\end{theorem}

We remark that by the work of S. Brendle in \cite{Bre}, the above result remains valid if we replace the assumption of nonnegative curvature operator by the weaker condition that $(M^n, g(t)) \times {\mathbb R}^2$ has nonnegative isotropic curvature.

Similar to the arguments in \cite{CaZh2014}, we first establish a local differential Harnack estimate which will lead to a global upper bound for our Harnack quantity. Then, refining the existing Harnack inequality with Hamilton's distance like function gives us the final result.

Since the last term in \eqref{eq F} has different signs for $b\geq 2$ and $1<b<2$, we treat these two cases separately.

\subsection{The case of $1<b<2$}
The local differential Harnack inequality in this case is as follows:
\begin{proposition}\label{local harnack p<1 1<b<2}
Let $(M^n, g_{ij}(t))$, $t\in[0,T]$, be a complete solution to the Ricci flow \eqref{rf} with bounded curvature and nonnegative curvature operator. If $u$ is a positive solution to \eqref{eq2} with $p\in(1-\frac{2(b-1)}{bn}, 1)$, then for $v=\frac{p}{p-1}u^{p-1}$, any point $O\in M$, any constants $R_0>0$, $\delta>0$, and $b\in(1,2)$, we have
\begin{align*}
\frac{|\nabla v|^2}{v}-b\frac{v_t}{v}-(b-1)\frac{R}{v}&\geq -\frac{b\alpha}{t}\left[1+\frac{1}{2}\sqrt{\frac{1+\delta}{\delta}}+C_1\frac{t\bar{v}_{max}}{R_0^2}+C_2(\delta)tR_{max}\right]
\end{align*}
on $\coprod_{t\in(0, T]}B_t(O, R_0)\times\{t\}$, where $C_1=(8n+128)(1-p)+\frac{32b\alpha p^2}{b-1}$, $C_2(\delta)=17-p+\sqrt{\frac{b\alpha(2-b)^2(1-p)}{2}+\delta(1-p)^2}$, $\alpha=\frac{bn(1-p)}{2-bn(1-p)}$, $\bar{v}_{max}=\sup_{\coprod_{t\in[0, T]}B_t(O, 2R_0)\times\{t\}}(-v)$, and $R_{max}=\sup_{M\times[0,T]}R(x,t)$.
\end{proposition}

\begin{proof}
From \eqref{eq F}, we have
\begin{align*}
\mathcal{L}(F)&=2p\nabla_i F\nabla_iv+[\frac{b-1}{-v}+1-p](\frac{\partial R}{\partial t}-2\nabla_i R\nabla_i v+2R_{ij}\nabla_i v \nabla_j v)\\
&\quad\ +2(1-p)|\nabla^2 v+\frac{b}{2}Rc|^2-\frac{(b-2)^2}{2}(1-p)|Rc|^2-\frac{1}{b}(y-bz)^2\\
&\quad\ +[(1-p)R+\frac{2(b-1)}{b}\frac{R}{-v}]F -\frac{b-1}{b}y^2+\frac{(b-1)(2-b)}{b}\cdot \frac{R}{v}y.
\end{align*}

Let $\eta(s)$ be a smooth function defined for $s\geq0$ such that
$\eta(s)=1$ for $0\leq s \leq \frac{1}{2}$, $\eta(s)=0$ for $s\geq 1$, $\eta(s)>0$ for $\frac{1}{2}<s<1$, $|\eta^{\prime}|\leq 16\eta^{\frac{1}{2}}$ and $\eta^{\prime\prime}\geq -16\eta\geq -16$. Define
\begin{equation}\label{eq phi}
\phi(x,t)=\eta\left(\frac{r(x,t)}{2R_0}\right)
\end{equation}
on $\coprod_{t\in[0, T]} B_{t}(O, 2R_0)\times\{t\}$, where $r(x,t)$ is the distance function from $O$ at time $t$.
It follows that
\begin{align*}
t\phi\mathcal{L}(t\phi F)&=t\phi^2F+t^2\phi F\phi_t-(p-1)t^2v\phi F\Delta\phi - 2(p-1)t^2v\phi\nabla_i\phi\nabla_i F + t^2\phi^2\mathcal{L}(F)\\
&= \phi(\tilde{y}-b\tilde{z})+t\phi_t(\tilde{y}-b\tilde{z})-(p-1)tv\Delta\phi(\tilde{y}-b\tilde{z}) - 2(p-1)t^2v\phi\nabla_i\phi\nabla_i F\\
&\quad\ +2p t^2\phi^2\nabla_i F\nabla_iv+t\phi^2[\frac{b-1}{-v}+1-p]Q - t\phi^2R[\frac{b-1}{-v}+1-p]\\
&\quad\ +2(1-p)t^2\phi^2|\nabla^2 v+Rc-\frac{2-b}{2}Rc|^2-\frac{(b-2)^2}{2}(1-p)t^2\phi^2|Rc|^2-\frac{1}{b}(\tilde{y}-b\tilde{z})^2\\
&\quad\ +[(1-p)R+\frac{2(b-1)}{b}\frac{R}{-v}]t\phi(\tilde{y}-b\tilde{z}) -\frac{b-1}{b}\tilde{y}^2+\frac{(b-1)(2-b)}{b}\cdot \frac{R}{v}t\phi\tilde{y},
\end{align*}
where $\tilde{y}=t\phi y$, $\tilde{z}=t\phi z$, and
\begin{equation} \label{eq Q}
Q=t\frac{\partial R}{\partial t}+R-2t\nabla_i R\nabla_i v+2tR_{ij}\nabla_i v \nabla_j v\geq0.
\end{equation}
Note that inequality (2.4) is a consequence of Theorem \ref{matrix harnack}.

It was shown in \cite{CaZh2014} that
$$|\frac{\partial \phi}{\partial t}|\leq 16R_{max},\qquad \Delta\phi\geq-\frac{8n}{R^2_0}, \qquad |\nabla \phi|\leq \frac{8}{R_0}\phi^{\frac{1}{2}}.$$
If $t\phi F\geq 0$ in $\coprod_{t\in[0, T]}B_t(O, 2R_0)\times\{t\}$, then the theorem is automatically true. Otherwise, since $t\phi F=0$ on the parabolic boundary of $\coprod_{t\in[0, T]}B_t(O, 2R_0)\times\{t\}$, we may assume that $t\phi F$ achieves a negative minimum at time $t_0>0$ and some interior point $x_0$. Thus, at $(x_0, t_0)$, we have
$$\tilde{y}-b\tilde{z}=t_0\phi F<0, \qquad F\nabla\phi=-\phi\nabla F,\qquad \textrm{and}\qquad \mathcal{L}(t\phi F)(x_0,t_0)\leq0.$$
Moreover, since
$$- 2p t_0^2\phi F \nabla_i\phi\nabla_i v\geq 2p t_0^2\phi F|\nabla\phi||\nabla v|\geq t_0\phi F\frac{16p}{R_0}(-\tilde{y})^{\frac{1}{2}}(-t_0v)^{\frac{1}{2}},$$
$$R+\Delta v=y-z,\qquad\ \textrm{and}\ \qquad -\tilde{y}\geq t\phi\frac{R}{-v},$$
by the Cauchy-Schwarz inequality, we have
\begin{align*}
0\geq &t_0\phi\mathcal{L}(t\phi F)\\
\geq & (\tilde{y}-b\tilde{z})+ 16t_0R_{max}(\tilde{y}-b\tilde{z}) + (8n+128)\frac{(p-1)t_0v}{R_0^2}(\tilde{y}-b\tilde{z}) +
\frac{16p}{R_0}(-\tilde{y})^{\frac{1}{2}}(-t_0v)^{\frac{1}{2}}(\tilde{y}-b\tilde{z})\\
\quad\ &-(1-p)t_0\phi R-(b-1)t_0\phi(\frac{R}{-v})+\frac{2}{b^2n(1-p)}[\tilde{y}-b\tilde{z}-(b-1)(-\tilde{y})-\frac{b(2-b)}{2}(1-p)t_0\phi R]^2\\
\quad\ &-\frac{(b-2)^2}{2}(1-p)t_0^2\phi^2R^2-\frac{1}{b}(\tilde{y}-b\tilde{z})^2+[(1-p)R+\frac{2(b-1)}{b}\frac{R}{-v}]t\phi(\tilde{y}-b\tilde{z})\\
\quad\ & -\frac{b-1}{b}\tilde{y}^2+\frac{(b-1)(2-b)}{b}\cdot t_0^2\phi^2(\frac{R}{-v})^2\\
\geq & (\tilde{y}-b\tilde{z})+ 16t_0R_{max}(\tilde{y}-b\tilde{z}) + (8n+128)\frac{(p-1)t_0v}{R_0^2}(\tilde{y}-b\tilde{z}) +
\frac{16p}{R_0}(-\tilde{y})^{\frac{1}{2}}(-t_0v)^{\frac{1}{2}}(\tilde{y}-b\tilde{z})\\
\quad\ &-(1-p)t_0\phi R +\frac{1}{b\alpha}(\tilde{y}-b\tilde{z})^2+ \frac{b-1}{b}[\frac{2(b-1)}{bn(1-p)}-1](-\tilde{y})^2-\frac{4(b-1)}{b^2n(1-p)}(-\tilde{y})(\tilde{y}-b\tilde{z})\\
\quad\ &-\frac{(b-2)^2}{2}(1-p)t_0^2\phi^2R^2+[(1-p)R+\frac{2(b-1)}{b}\frac{R}{-v}]t_0\phi(\tilde{y}-b\tilde{z})\\
\quad\ & +\frac{(b-1)(2-b)}{b}\cdot t^2\phi^2(\frac{R}{-v})^2-(b-1)t\phi(\frac{R}{-v}).
\end{align*}
According to the assumption on $p$, we have $\frac{2(b-1)}{bn(1-p)}-1\geq0$. Moreover, if we choose $\beta=\frac{2(b-1)}{b\alpha}$, then we have
\begin{align*}
0\geq & (\tilde{y}-b\tilde{z})+ (16+(1-p))t_0R_{max}(\tilde{y}-b\tilde{z}) + [(8n+128)(1-p)+\frac{64p^2}{\beta}]\frac{t_0\bar{v}_{max}}{R_0^2}(\tilde{y}-b\tilde{z})
\\
\quad\ &-(1-p)t_0\phi R +\frac{1}{b\alpha}(\tilde{y}-b\tilde{z})^2-(\frac{2(b-1)}{b\alpha}-\beta)(-\tilde{y})(\tilde{y}-b\tilde{z})\\
\quad\ &-\frac{(b-2)^2}{2}(1-p)t_0^2\phi^2R^2+\frac{(b-1)^2}{b\alpha}\cdot t_0^2\phi^2(\frac{R}{-v})^2-(b-1)t_0\phi(\frac{R}{-v})\\
\geq & \frac{1}{b\alpha}(\tilde{y}-b\tilde{z})^2+\left[1+(17-p)t_0R_{max}+[(8n+128)(1-p)+\frac{64p^2}{\beta}]\frac{t_0\bar{v}_{max}}{R_0^2}\right](\tilde{y}-b\tilde{z})\\
\quad\ & -\frac{(b-2)^2}{2}(1-p)t_0^2R_{max}^2-(1-p)t_0R_{max}-\frac{b\alpha}{4}.
\end{align*}
This implies that
\begin{align*}
\tilde{y}-b\tilde{z}&\geq -b\alpha\left[1+(17-p)t_0R_{max}+[(8n+128)(1-p)+\frac{64p^2}{\beta}]\frac{t_0\bar{v}_{max}}{R_0^2}\right]\\
&\quad\ -\sqrt{b\alpha\left[\frac{(b-2)^2}{2}(1-p)t_0^2R_{max}^2+(1-p)t_0R_{max}+\frac{b\alpha}{4}\right]}.
\end{align*}
Now Proposition \ref{local harnack p<1 1<b<2} follows easily.
\end{proof}

If $u$ is bounded, then letting $R_0\rightarrow \infty$, we get a global differential Harnack inequality.
\begin{corollary}\label{noncompact harnack p<1 b<2}
Let $(M^n, g_{ij}(t))$, $t\in[0, T]$, be a complete solution to the Ricci flow \eqref{rf} with bounded curvature and nonnegative curvature operator. If $u$ is a smooth positive bounded solution to \eqref{eq FDE} with $1-\frac{2(b-1)}{bn}<p<1$, then for $v=\frac{p}{p-1}u^{p-1}$, any constants $\delta>0$, and $b\in(1,2)$, we have
\begin{align*}
\frac{|\nabla v|^2}{v}-b\frac{v_t}{v}-(b-1)\frac{R}{v}&\geq -\frac{b\alpha}{t}\left[1+\frac{1}{2}\sqrt{\frac{1+\delta}{\delta}}+C_2(\delta)tR_{max}\right]
\end{align*}
on $M\times(0,T]$, where $C_2(\delta)$, $\alpha$, and $R_{max}$ are the same constants as in Proposition \ref{local harnack p<1 1<b<2}.
\end{corollary}

We shall need the following lemma about some distance-like function constructed by Hamilton (see e.g. \cite{Ham1993jdg} and \cite{CaZh2006}),
\begin{lemma}\label{distance function}
Let $g_{ij}(t)$, $t\in[0, T]$, be a complete solution to the Ricci flow on $M^n$ with bounded curvature tensor. Then, there exists a smooth function $f(x)$ on $M$  and a positive constant $C>0$ such that $f\geq1$,
$f(x)\rightarrow\infty$ as $d_0(x,O)\rightarrow\infty$ (for any fixed point $O\in M$),
$$|\nabla f|_{g(t)}\leq C,\qquad and \qquad |\nabla\nabla f|_{g(t)}\leq C$$
on $M\times[0,T]$.
\end{lemma}

Now using the function $f$ in the above Lemma and a method of Hamilton, we are able to refine the differential Harnack inequality in Corollary \ref{noncompact harnack p<1 b<2}.

\begin{theorem}\label{refined noncompact harnack p<1 b<2}
Let $(M^n, g_{ij}(t))$, $t\in[0, T]$, be a complete solution to the Ricci flow \eqref{rf} with bounded curvature and nonnegative curvature operator. If $u$ is a smooth positive bounded solution to \eqref{eq2}  with $1-\frac{2(b-1)}{bn}<p<1$, then for $v=\frac{p}{p-1}u^{p-1}$, and $b\in(1,2)$, we have
$$\frac{|\nabla v|^2}{v}-b\frac{v_t}{v}-(b-1)\frac{R}{v}+KR_{max}+\frac{b\alpha}{t}\geq 0$$
on $M\times(0,T]$, where $\alpha=\frac{bn(1-p)}{2-bn(1-p)}$, $K=\max\{\frac{(b\alpha+1)(1-p)}{2}, b\alpha(1-p)+(2-b)\sqrt{\frac{b\alpha(n-1)(1-p)}{2n}}\}$, and $R_{max}=\sup_{M\times[0,T]}R(x,t)$.
\end{theorem}

\begin{proof}
From \eqref{eq F}, we have
\begin{align*}
\mathcal{L}(F)&=2p\nabla_i F\nabla_iv+[\frac{b-1}{-v}+1-p](\frac{\partial R}{\partial t}-2\nabla_i R\nabla_i v+2R_{ij}\nabla_i v \nabla_j v)\\
&\quad\ +2(1-p)|\nabla^2 v+\frac{b}{2}Rc|^2-\frac{(b-2)^2}{2}(1-p)|Rc|^2-\frac{1}{b}F^2\\
&\quad\ +[(1-p)R+\frac{2(b-1)}{b}\frac{R}{-v}]F -\frac{b-1}{b}y^2+\frac{(b-1)(2-b)}{b}\cdot \frac{R}{v}y.
\end{align*}

Let $H=t(F+K)+d$, then
\begin{align*}
t\mathcal{L}(H)&=t(F+K)+t^2\mathcal{L}(F)\\
&= H-d+2p t\nabla_i H\nabla_iv+t[\frac{b-1}{-v}+1-p]Q-tR[\frac{b-1}{-v}+1-p]\\
&\quad\ +2(1-p)t^2|\nabla^2 v+Rc-\frac{2-b}{2}Rc|^2-\frac{(b-2)^2}{2}(1-p)t^2|Rc|^2-\frac{1}{b}(H-tK-d)^2\\
&\quad\ +[(1-p)R+\frac{2(b-1)}{b}\frac{R}{-v}]t(H-tK-d) -\frac{b-1}{b}(ty)^2+\frac{(b-1)(2-b)}{b}\cdot \frac{t^2R}{v}y\\
&\geq H-d+2p t\nabla_i H\nabla_iv-tR[\frac{b-1}{-v}+1-p]\\
&\quad\ +\frac{2}{b^2n(1-p)}\left[H-tK-d-(b-1)(-ty)-\frac{b(2-b)}{2}(1-p)tR\right]^2\\
&\quad\ -\frac{(b-2)^2}{2}(1-p)t^2R^2-\frac{1}{b}(H-tK-d)^2+[(1-p)R+\frac{2(b-1)}{b}\frac{R}{-v}]t(H-tK-d) \\
&\quad\ -\frac{b-1}{b}(ty)^2+\frac{(b-1)(2-b)}{b}\cdot (\frac{tR}{-v})(-ty)\\
&\geq H-d+2p t\nabla_i H\nabla_iv+(b-1)(-\frac{2d}{b}-1)\frac{tR}{-v}+[\frac{2(2-b)d}{bn(1-p)}-d-1](1-p)tR\\
&\quad\ +\frac{1}{b\alpha}(H-tK-d)^2+\frac{b-1}{b}[\frac{2(b-1)}{bn(1-p)}-1](-ty)^2-\frac{(b-2)^2(n-1)}{2n}(1-p)t^2R^2\\
&\quad\ -\frac{4(b-1)}{b^2n(1-p)}(-ty)(H-tK-d)+\frac{2(b-1)}{b}\frac{tR}{-v}(H-tK)\\
&\quad\ -[\frac{2(2-b)}{bn(1-p)}-1](1-p)tR(H-tK).
\end{align*}

Let $\tilde{H}=H+\epsilon\psi$, where $\psi=e^{At}f$ for some constant $A>0$ to be determined, and $f$ being the function in Lemma \ref{distance function}. Then
\begin{align*}
& t\mathcal{L}(\tilde{H})= t\mathcal{L}(H) + tA\epsilon\psi -(p-1)vt\epsilon e^{At}\Delta f.
\end{align*}

From Corollary \ref{noncompact harnack p<1 b<2}, we know that $\tilde{H}>0$ at $t=0$ and outside a fixed compact subset of $M$ for $t\in(0,T]$. Suppose that $\tilde{H}$ reaches $0$ for the first time at some point $x_0\in M$ when $t=t_0>0$. Then we have at $(x_0, t_0)$,
$$H=-\epsilon\psi < 0,\qquad \nabla\tilde{H}=0,\qquad\textrm{and}\qquad \mathcal{L}(\tilde{H})\leq 0.$$
Setting $d=b\alpha$, and $\beta=\frac{2(b-1)}{bn(1-p)}-1>0$ we have
\begin{align*}
0\geq & t_0\mathcal{L}(\tilde{H})\\
\geq &  -\epsilon\psi-b\alpha-2p t_0\epsilon e^{At}\nabla_i f\nabla_iv+[\frac{2(2-b)b\alpha}{bn(1-p)}-b\alpha-1](1-p)t_0R+\frac{1}{b\alpha}(\epsilon\psi+t_0K+b\alpha)^2\\
\quad\ &+\frac{\beta(b-1)}{b}(-t_0y)^2-\frac{(b-2)^2(n-1)}{2n}(1-p)t_0^2R^2\\
\quad\ &+[\frac{2(2-b)}{bn(1-p)}-1](1-p)t_0R(\epsilon\psi+t_0K) + t_0A\epsilon\psi -(p-1)vt_0\epsilon e^{At_0}\Delta f\\
\end{align*}
\begin{align*}
\geq& -Ct_0\epsilon\psi|\nabla v|+[\frac{2(2-b)b\alpha}{bn(1-p)}-b\alpha-1](1-p)t_0R+\frac{1}{b\alpha}(\epsilon\psi)^2+\frac{1}{b\alpha}t_0^2K^2+\epsilon\psi+\frac{2}{b\alpha}t_0K\epsilon\psi\\
\quad\ &+2t_0K+\frac{\beta(b-1)}{b}(\frac{t_0|\nabla v|^2}{-v})^2-\frac{(b-2)^2(n-1)}{2n}(1-p)t_0^2R^2\\
\quad\ &+[\frac{2(2-b)}{bn(1-p)}-1](1-p)t_0R(\epsilon\psi+t_0K) + t_0A\epsilon\psi -C\bar{v}_{max}t_0\epsilon\psi\\
\geq& -\delta\epsilon\psi\frac{t_0|\nabla v|^2}{-v}-\frac{C}{4\delta}t_0(-v)\epsilon\psi+\frac{1}{b\alpha}(\epsilon\psi)^2+ \frac{\beta(b-1)}{b}(\frac{t_0|\nabla v|}{-v})^2\\
\quad\ &-[b\alpha+1](1-p)t_0R+2t_0K + \frac{1}{b\alpha}t_0^2K^2-\frac{(b-2)^2(n-1)}{2n}(1-p)t_0^2R^2\\
\quad\ &-(1-p)t_0^2RK + \left(A+\frac{2}{b\alpha}K+[\frac{2(2-b)}{bn(1-p)}-1](1-p)R-C\bar{v}_{max}\right)t_0\epsilon\psi.
\end{align*}

We may choose a suitable $\delta$ so that $$-\delta\epsilon\psi\frac{t_0|\nabla v|^2}{-v}+\frac{1}{b\alpha}(\epsilon\psi)^2+ \frac{\beta(b-1)}{b}(\frac{t_0|\nabla v|^2}{-v})^2\geq0.$$
Notice that by choosing $K=\max\{\frac{(b\alpha+1)(1-p)}{2}, b\alpha(1-p)+(2-b)\sqrt{\frac{b\alpha(n-1)(1-p)}{2n}}\}R_{max}$, we can make both
$$-[b\alpha+1](1-p)R+2K\geq0,$$
and
$$\frac{1}{b\alpha}K^2-\frac{(b-2)^2(n-1)}{2n}(1-p)R^2-(1-p)RK\geq0.$$

Hence if we choose $A>(\frac{C}{4\delta}+C)\bar{v}_{max}+(1-p)R_{max}$, then we have
\begin{align*}
0\geq & t\mathcal{L}(\tilde{H})\\
\geq& \left[A+\frac{2}{b\alpha}K-(1-p)R-(C+\frac{C}{4\delta})\bar{v}_{max}\right]t_0\epsilon\psi\\
>&0,
\end{align*}
a contradiction.
Therefore, $\tilde{H}>0$ for all $t\in[0,T]$. Letting $\epsilon\rightarrow 0$, we get $$H\geq0.$$
\end{proof}

As in  \cite{CaZh2014}, consequently we have the following Harnack inequalities.
\begin{corollary}\label{cor5}
Suppose that $(M, g_{ij}(t))$, $t\in[0, T]$, is a complete solution of the Ricci flow \eqref{rf} with bounded curvature and nonnegative curvature operator, and $u$ is a bounded smooth positive solution to \eqref{eq2} for $1-\frac{2(b-1)}{bn}<p<1$ and $b\in(1,2)$. Let $\bar{v}=-v=\frac{p}{1-p}u^{p-1}$, if $\bar{v}_{min}=\inf_{M\times[0,T]}\bar{v}>0$, then for any points $x_1$,$ x_2\in M$ and $0<t_1<t_2$, we have
\begin{equation*}\label{eq cor5}
\bar{v}(x_2, t_2)\leq \bar{v}(x_1, t_1)\cdot (\frac{t_2}{t_1})^{\alpha}exp\left( \frac{\Gamma}{\bar{v}_{min}}+\frac{KR_{max}}{b}(t_2-t_1)\right),
\end{equation*}
where $\alpha$, $K$ and $R_{max}$ are the same constants as in Theorem \ref{refined noncompact harnack p<1 b<2}, and $\Gamma=\inf_{\gamma}\int_{t_1}^{t_2}(\frac{b-1}{b}R+\frac{b}{4}\left|\frac{d\gamma}{d\tau}\right|^2_{g_{ij}(\tau)})d\tau$ with the infimum taking over all smooth curves $\gamma(\tau)$ in $M$, $\tau\in[t_1, t_2]$, with $\gamma(t_1)=x_1$ and $\gamma(t_2)=x_2$.

\end{corollary}

\begin{corollary}\label{cor6}
Suppose that $(M, g_{ij}(t))$, $t\in[0, T]$, is a complete solution of the Ricci flow \eqref{rf} with bounded curvature and nonnegative curvature operator, and $u$ is a bounded smooth positive solution to \eqref{eq2} for $1-\frac{2(b-1)}{bn}<p<1$ and $b\in(1,2)$. Let $\bar{v}=-v=\frac{p}{1-p}u^{p-1}$, then for any points $x_1$,$ x_2\in M$ and $0<t_1<t_2$, we have
\begin{equation*}\label{eq cor6}
\bar{v}(x_2, t_2)-\bar{v}(x_1, t_1)\leq \alpha \bar{v}_{max}\ln \frac{t_2}{t_1} + [\frac{b-1}{b}+\frac{K}{b}\bar{v}_{max}]R_{max}(t_2-t_1)+\frac{b}{4}\frac{d_{t_1}^2(x_1,x_2)}{t_2-t_1}
\end{equation*}
where $\alpha$, $K$ and $R_{max}$ are the same constants as in Theorem \ref{refined noncompact harnack p<1 b<2}, and $\bar{v}_{max}=\sup_{M\times[0,T]}\bar{v}$.
\end{corollary}
\bigskip

\subsection{The case of $b\geq 2$}
In this case, we first have

\begin{proposition}\label{local harnack p<1 b>2}
Let $(M^n,  g_{ij} (t))$, $t\in[0, T]$, be a complete solution to the Ricci flow \eqref{rf} with bounded curvature and nonnegative curvature operator. If $u$ is a smooth positive solution to \eqref{eq2} with $p\in(1-\frac{2}{bn},1)$ and $b\geq2$, then for $v=\frac{p}{p-1}u^{p-1}$, any point $O\in M$, any constants $R_0>0$, $\delta_1\in(0,1)$ and $\delta_2>0$, we have
\begin{align*}
\frac{|\nabla v|^2}{v}-b\frac{v_t}{v}-(b-1)\frac{R}{v}&\geq -\frac{b\alpha}{t}\left[1+\frac{1}{2}\sqrt{\frac{1}{1-\delta_1}+\delta_2}+C_1(\delta_1,\delta_2)R_{max}+C_2\frac{t\bar{v}_{max}}{R_0^2}\right]
\end{align*}
on $\coprod_{t\in(0, T]}B_t(O, R_0)\times\{t\}$, where $$C_1(\delta_1,\delta_2)=17-p+\frac{2(b-2)}{bn}+\sqrt{b\alpha\left[\frac{(2-b)^2(1-p)}{2}+\frac{\alpha(b-2)^2}{bn^2\delta_1}\right]+\frac{(1-p)^2}{\delta_2}},$$ $$C_2=(8n+128)(1-p)+\frac{32b\alpha p^2}{b-1},$$
$\alpha=\frac{bn(1-p)}{2-bn(1-p)}$, $\bar{v}_{max}=\max_{\coprod_{t\in[0, T]}B_t(O, 2R_0)\times\{t\}}\{-v\}$ and $R_{max}=\sup_{M\times[0,T]}R$.
\end{proposition}

\begin{proof}
Let $\phi(x,t)$ be the same cut-off function as in \eqref{eq phi}. It follows from \eqref{eq F} that
\begin{align*}
t\phi\mathcal{L}(t\phi F)&= \phi(\tilde{y}-b\tilde{z})+t\phi_t(\tilde{y}-b\tilde{z})-(p-1)tv\Delta\phi(\tilde{y}-b\tilde{z}) - 2(p-1)t^2v\phi\nabla_i\phi\nabla_i F\\
&\quad\ +2p t^2\phi^2\nabla_i F\nabla_iv+t\phi^2[\frac{b-1}{-v}+1-p]Q - t\phi^2R[\frac{b-1}{-v}+1-p]\\
&\quad\ +2(1-p)t^2\phi^2|\nabla^2 v+Rc+\frac{b-2}{2}Rc|^2-\frac{(b-2)^2}{2}(1-p)t^2\phi^2|Rc|^2-\frac{1}{b}(\tilde{y}-b\tilde{z})^2\\
&\quad\ +[(1-p)R+\frac{2(b-1)}{b}\frac{R}{-v}]t\phi(\tilde{y}-b\tilde{z}) -\frac{b-1}{b}\tilde{y}^2-\frac{(b-1)(b-2)}{b}\cdot \frac{R}{v}t\phi\tilde{y},
\end{align*}
where $\tilde{y}=t\phi y$, $\tilde{z}=t\phi z$, and $Q\geq 0$ is the quantity in \eqref{eq Q}.

If $t\phi F\geq 0$ in $\coprod_{t\in[0, T]}B_t(O, 2R_0)\times\{t\}$, then we are done. Otherwise, since $t\phi F=0$ on the parabolic boundary of $\coprod_{t\in[0, T]}B_t(O, 2R_0)\times\{t\}$, we may assume that $t\phi F$ achieves a negative minimum for the first time at $t_0>0$ and some interior point $x_0$. Thus, at $(x_0, t_0)$, we have
$$\tilde{y}-b\tilde{z}=t_0\phi F<0, \qquad \nabla\phi=-\phi\nabla F,\qquad \textrm{and}\qquad \mathcal{L}(t\phi F)(x_0,t_0)\leq0.$$
Moreover, since
$$- 2p t_0^2\phi F \nabla_i\phi\nabla_i v\geq 2p t_0^2\phi F|\nabla\phi||\nabla v|\geq t_0\phi F\frac{16p}{R_0}(-\tilde{y})^{\frac{1}{2}}(-t_0v)^{\frac{1}{2}},$$
we get
\begin{align*}
0\geq &t_0\phi\mathcal{L}(t\phi F)\\
\geq & (\tilde{y}-b\tilde{z})+ 16t_0R_{max}(\tilde{y}-b\tilde{z}) + (8n+128)\frac{(p-1)t_0v}{R_0^2}(\tilde{y}-b\tilde{z})+
\frac{16p}{R_0}(-\tilde{y})^{\frac{1}{2}}(-t_0v)^{\frac{1}{2}}(\tilde{y}-b\tilde{z})\\
\quad\ &-(1-p)t_0\phi R-(b-1)t_0\phi(\frac{R}{-v})+\frac{2}{b^2n(1-p)}[\tilde{y}-b\tilde{z}-(b-1)(-\tilde{y})+\frac{b(b-2)}{2}(1-p)t_0\phi R]^2\\
\quad\ & -\frac{(b-2)^2}{2}(1-p)t_0^2\phi^2R^2-\frac{1}{b}(\tilde{y}-b\tilde{z})^2+[(1-p)R+\frac{2(b-1)}{b}\frac{R}{-v}]t_0\phi(\tilde{y}-b\tilde{z})\\
\quad\ &-\frac{b-1}{b}\tilde{y}^2-\frac{(b-1)(b-2)}{b}\cdot (-\tilde{y})^2\\
\geq & (\tilde{y}-b\tilde{z})+ 16t_0R_{max}(\tilde{y}-b\tilde{z}) + (8n+128)\frac{(p-1)t_0v}{R_0^2}(\tilde{y}-b\tilde{z}) +
\beta(-\tilde{y})(\tilde{y}-b\tilde{z})\\
\quad\ &+\frac{64p^2}{\beta R^2_0}(-t_0v)(\tilde{y}-b\tilde{z})-(1-p)t_0\phi R +\frac{1}{b\alpha}(\tilde{y}-b\tilde{z})^2 +\frac{(b-1)^2}{b\alpha}(-\tilde{y})^2-(b-1)t\phi(\frac{R}{-v})\\
\quad\ &-\frac{2(b-1)(b-2)}{bn}t_0\phi R(-\tilde{y})-\frac{4(b-1)}{b^2n(1-p)}(-\tilde{y})(\tilde{y}-b\tilde{z})+\frac{2(b-2)}{bn(1-p)}(1-p)t_0\phi R(\tilde{y}-b\tilde{z})\\
\quad\ &-\frac{(b-2)^2}{2}(1-p)t_0^2\phi^2R^2+[(1-p)R+\frac{2(b-1)}{b}\frac{R}{-v}]t_0\phi(\tilde{y}-b\tilde{z}).
\end{align*}
If we choose $\beta=\frac{2(b-1)}{b\alpha}$ and $0<\delta_1<1$, then we have
\begin{align*}
0\geq & (\tilde{y}-b\tilde{z})+ (17-p+\frac{2(b-2)}{bn})t_0R_{max}(\tilde{y}-b\tilde{z}) + [(8n+128)(1-p)+\frac{64p^2}{\beta}]\frac{t_0\bar{v}_{max}}{R_0^2}(\tilde{y}-b\tilde{z})
\\
\quad\ &-(1-p)t_0\phi R +\frac{1}{b\alpha}(\tilde{y}-b\tilde{z})^2-(\frac{2(b-1)}{b\alpha}-\beta)(-\tilde{y})(\tilde{y}-b\tilde{z})+\frac{(b-1)^2}{b\alpha}(1-\delta_1)\cdot t^2\phi^2(\frac{R}{-v})^2\\
\quad\ &-(b-1)t\phi(\frac{R}{-v})-\left[\frac{(b-2)^2}{2}(1-p)+\frac{\alpha(b-2)^2}{bn^2\delta_1}\right]t_0^2\phi^2R^2\\
\geq & \frac{1}{b\alpha}(\tilde{y}-b\tilde{z})^2+\left[1+(17-p+\frac{2(b-2)}{bn})t_0R_{max}+[(8n+128)(1-p)+\frac{64p^2}{\beta}]\frac{t_0\bar{v}_{max}}{R_0^2}\right](\tilde{y}-b\tilde{z})\\
\quad\ & -\left[\frac{(b-2)^2}{2}(1-p)+\frac{\alpha(b-2)^2}{bn^2\delta_1}\right]t_0^2R_{max}^2-(1-p)t_0R_{max}-\frac{b\alpha}{4(1-\delta_1)}.
\end{align*}

Therefore, we have
\begin{align*}
\tilde{y}-b\tilde{z}&\geq -b\alpha\left[1+(17-p+\frac{2(b-2)}{bn})t_0R_{max}+[(8n+128)(1-p)+\frac{64p^2}{\beta}]\frac{t_0\bar{v}_{max}}{R_0^2}\right]\\
&\quad\ -\sqrt{b\alpha\left(\left[\frac{(b-2)^2}{2}(1-p)+\frac{\alpha(b-2)^2}{bn^2\delta_1}\right]t_0^2R_{max}^2+(1-p)t_0R_{max}+\frac{b\alpha}{4(1-\delta_1)}\right)}.
\end{align*}
This finishes the proof.
\end{proof}

Assuming further that $u$ is bounded and letting $R_0\rightarrow \infty$, we have

\begin{corollary}\label{noncompact harnack p<1 b>2}
Let $(M^n, g_{ij}(t))$, $t\in[0, T]$, be a complete solution to the Ricci flow \eqref{rf} with bounded curvature and nonnegative curvature operator. If $u$ is a bounded smooth positive solution to \eqref{eq2} with $p\in(1-\frac{2}{bn},1)$ and $b\geq2$, then for $v=\frac{p}{p-1}u^{p-1}$ , and any constants $\delta_1\in(0,1)$ and $\delta_2>0$, we have
\begin{align*}
\frac{|\nabla v|^2}{v}-b\frac{v_t}{v}-(b-1)\frac{R}{v}&\geq -\frac{b\alpha}{t}\left[1+\frac{1}{2}\sqrt{\frac{1}{1-\delta_1}+\delta_2}+C_1(\delta_1,\delta_2)R_{max}\right]
\end{align*}
on $M\times(0,T]$, where $\alpha$, $C_1(\delta_1,\delta_2)$ and $R_{max}$ are the same constants as in Proposition \ref{local harnack p<1 b>2}.
\end{corollary}

Using Hamilton's method, again we can obtain the following refined estimate.

\begin{theorem}\label{refined noncompact harnack p<1 b>2}
Let $(M^n, g_{ij}(t))$, $t\in[0, T]$, be a complete solution to the Ricci flow \eqref{rf} with bounded curvature and nonnegative curvature operator. If $u$ is a bounded smooth positive solution to \eqref{eq2} with $p\in(1-\frac{2}{bn},1)$ and $b\geq2$, then for $v=\frac{p}{p-1}u^{p-1}$, we have
$$\frac{|\nabla v|^2}{v}-b\frac{v_t}{v}-(b-1)\frac{R}{v}+KR_{max}+\frac{b\alpha}{t}\geq 0$$
on $M\times(0,T]$, where $\alpha$ and $R_{max}$ are the same constants as in Proposition \ref{local harnack p<1 b>2}, and $$K=\max\{[\frac{2(b-2)\alpha}{n(1-p)}+b\alpha+1]\frac{(1-p)}{2},\,  b\alpha[\frac{2(b-2)}{bn(1-p)}+1](1-p)+(b-2)\sqrt{\frac{b\alpha(n-1)(1-p)}{2n}}\}.$$
\end{theorem}

\begin{proof}
Let $H=t(F+K)+d$ with $d=b\alpha$ and $K>0$ to be determined. From \eqref{eq F}, one has
\begin{align*}
t\mathcal{L}(H)\geq &H-d+2p t\nabla_i H\nabla_iv-tR[\frac{b-1}{-v}+1-p]-\frac{b-1}{b}(ty)^2-\frac{(b-1)(b-2)}{b}\cdot(ty)^2\\
\quad\ &+\frac{2}{b^2n(1-p)}\left[H-tK-d-(b-1)(-ty)+\frac{b(b-2)}{2}(1-p)tR\right]^2\\
\quad\ &-\frac{(b-2)^2}{2}(1-p)t^2R^2-\frac{1}{b}(H-tK-d)^2+[(1-p)R+\frac{2(b-1)}{b}\frac{R}{-v}]t(H-tK-d) \\
\geq &H-d+2p t\nabla_i H\nabla_iv+(b-1)(-\frac{2d}{b}-1)\frac{tR}{-v}-[\frac{2(b-2)d}{bn(1-p)}+d+1](1-p)tR\\
\quad\ &+\frac{1}{b\alpha}(H-tK-d)^2+\frac{(b-1)^2}{b\alpha}(-ty)^2-\frac{(b-2)^2(n-1)}{2n}(1-p)t^2R^2\\
\quad\ &-\frac{4(b-1)}{b^2n(1-p)}(-ty)(H-tK-d)+\frac{2(b-1)}{b}\frac{tR}{-v}(H-tK)\\
\quad\ &+[\frac{2(b-2)}{bn(1-p)}+1](1-p)tR(H-tK)-\frac{2(b-1)(b-2)}{bn}tR(-ty).
\end{align*}

Let $\tilde{H}=H+\epsilon\psi$, where $\psi=e^{At}f$ for some constant $A>0$ to be determined, and $f$ being the function in Lemma \ref{distance function}. From Corollary \ref{noncompact harnack p<1 b>2}, we know that $\tilde{H}>0$ at $t=0$ and outside a fixed compact subset of $M$ for $t\in(0,T]$. Suppose that $\tilde{H}$ reaches $0$ for the first time at some point $x_0\in M$ when $t=t_0>0$. Then we have at $(x_0, t_0)$,
$$H=-\epsilon\psi<0,\qquad \nabla\tilde{H}=0,\qquad\textrm{and}\qquad \mathcal{L}(\tilde{H})\leq 0.$$
Hence, we get
\begin{align*}
0\geq& t_0\mathcal{L}(\tilde{H})\\\geq& -Ct_0\epsilon\psi|\nabla v|-[\frac{2(b-2)b\alpha}{bn(1-p)}+b\alpha+1](1-p)t_0R+\frac{1}{b\alpha}(\epsilon\psi)^2+\frac{1}{b\alpha}t_0^2K^2+\epsilon\psi\\
\quad\ &+\frac{2}{b\alpha}t_0K\epsilon\psi+2t_0K+\frac{(b-1)^2}{b\alpha}(\frac{t_0|\nabla v|^2}{-v})^2-\frac{(b-2)^2(n-1)}{2n}(1-p)t_0^2R^2\\
\quad\ &-[\frac{2(b-2)}{bn(1-p)}+1](1-p)t_0R(\epsilon\psi+t_0K) + t_0A\epsilon\psi -C\bar{v}_{max}t_0\epsilon\psi\\
\quad\ & +[\frac{4(b-1)}{b^2n(1-p)}K-\frac{2(b-1)(b-2)}{bn}R](-t_0^2y)\\
\geq& -\delta\epsilon\psi\frac{t_0|\nabla v|^2}{-v}-\frac{C}{4\delta}t_0(-v)\epsilon\psi+\frac{1}{b\alpha}(\epsilon\psi)^2+ \frac{(b-1)^2}{b\alpha}(\frac{t_0|\nabla v|^2}{-v})^2\\
\quad\ &-[\frac{2(b-2)\alpha}{n(1-p)}+b\alpha+1](1-p)t_0R+2t_0K + \frac{1}{b\alpha}t_0^2K^2-\frac{(b-2)^2(n-1)}{2n}(1-p)t_0^2R^2\\
\quad\ &-[\frac{2(b-2)}{bn(1-p)}+1](1-p)t_0^2RK + \left(A+\frac{2}{b\alpha}K-[\frac{2(2-b)}{bn(1-p)}+1](1-p)R-C\bar{v}_{max}\right)t_0\epsilon\psi\\
\quad\ &+[\frac{4(b-1)}{b^2n(1-p)}K-\frac{2(b-1)(b-2)}{bn}R](-t_0^2y).
\end{align*}

We may choose appropriate $\delta$ so that $$-\delta\epsilon\psi\frac{t_0|\nabla v|^2}{-v}+\frac{1}{b\alpha}(\epsilon\psi)^2+ \frac{(b-1)^2}{b\alpha}(\frac{t_0|\nabla v|^2}{-v})^2\geq0.$$
Notice that by choosing $$K=\max\{[\frac{2(b-2)\alpha}{n(1-p)}+b\alpha+1]\frac{(1-p)}{2}, b\alpha[\frac{2(b-2)}{bn(1-p)}+1](1-p)+(b-2)\sqrt{\frac{b\alpha(n-1)(1-p)}{2n}}\}R_{max},$$ we can make
$$\frac{4(b-1)}{b^2n(1-p)}K-\frac{2(b-1)(b-2)}{bn}R\geq0,$$
$$-[\frac{2(b-2)\alpha}{n(1-p)}+b\alpha+1](1-p)R+2K\geq0,$$
and
$$\frac{1}{b\alpha}K^2-\frac{(b-2)^2(n-1)}{2n}(1-p)R^2-[\frac{2(b-2)}{bn(1-p)}+1](1-p)RK\geq0.$$

Hence if we furthermore choose $A>(\frac{C}{4\delta}+C)\bar{v}_{max}+(1-p)R_{max}$, then we get
\begin{align*}
0\geq  t_0\mathcal{L}(\tilde{H})\geq \left[A-(1-p)R-(C+\frac{C}{4\delta})\bar{v}_{max}\right]t_0\epsilon\psi>0,
\end{align*}
a contradiction.
Therefore, $\tilde{H}>0$ for all $t\in[0,T]$. Letting $\epsilon\rightarrow 0$, it follows that $$H\geq0.$$
\end{proof}

Integrating the differential Harnack inequality in the above theorem along space-time paths, we obtain

\begin{corollary}\label{cor5}
Suppose that $(M, g_{ij}(t))$, $t\in[0, T]$, is a complete solution of the Ricci flow with bounded curvature and nonnegative curvature operator, and $u$ is a bounded smooth positive solution to \eqref{eq2} for $1-\frac{2}{bn}<p<1$ and $b\geq2$. Let $\bar{v}=-v=\frac{p}{1-p}u^{p-1}$, if $\bar{v}_{min}=\inf_{M\times[0,T]}\bar{v}>0$, then for any points $x_1$,$ x_2\in M$ and $0<t_1<t_2$, we have
\begin{equation}\label{eq cor5}
\bar{v}(x_2, t_2)\leq \bar{v}(x_1, t_1)\cdot (\frac{t_2}{t_1})^{\alpha}exp\left( \frac{\Gamma}{\bar{v}_{min}}+\frac{KR_{max}}{b}(t_2-t_1)\right),
\end{equation}
where $\alpha$, $K$ and $R_{max}$ are the same constants as in Theorem \ref{refined noncompact harnack p<1 b>2}, and $\Gamma=\inf_{\gamma}\int_{t_1}^{t_2}(\frac{b-1}{b}R+\frac{b}{4}\left|\frac{d\gamma}{d\tau}\right|^2_{g_{ij}(\tau)})d\tau$ with the infimum taking over all smooth curves $\gamma(\tau)$ in $M$, $\tau\in[t_1, t_2]$, with $\gamma(t_1)=x_1$ and $\gamma(t_2)=x_2$.

\end{corollary}

\begin{corollary}\label{cor6}
Suppose that $(M, g_{ij}(t))$, $t\in[0, T]$, is a complete solution of the Ricci flow \eqref{rf} with bounded curvature and nonnegative curvature operator, and $u$ is a bounded smooth positive solution to \eqref{eq2} for $1-\frac{2}{bn}<p<1$ and $b\geq2$. Let $\bar{v}=-v=\frac{p}{1-p}u^{p-1}$, then for any points $x_1$,$ x_2\in M$ and $0<t_1<t_2$, we have
\begin{equation}\label{eq cor6}
\bar{v}(x_2, t_2)-\bar{v}(x_1, t_1)\leq \alpha \bar{v}_{max}\ln \frac{t_2}{t_1} + [\frac{b-1}{b}+\frac{K}{b}\bar{v}_{max}]R_{max}(t_2-t_1)+\frac{b}{4}\frac{d_{t_1}^2(x_1,x_2)}{t_2-t_1}
\end{equation}
where $\alpha$, $K$ and $R_{max}$ are the same constants as in Theorem \ref{refined noncompact harnack p<1 b>2}, and $\bar{v}_{max}=\sup_{M\times[0,T]}\bar{v}$.
\end{corollary}

\bigskip

\section{Li-Yau-Hamilton Estimate for the conjugate Heat equation}
In this section, we consider the conjugate heat equation
\begin{equation}\label{conjugate heat equation}
\frac{\partial u}{\partial t}=-\Delta u+Ru
\end{equation}
under the Ricci flow \eqref{rf} and derive Li-Yau type estimates for positive solutions to \eqref{conjugate heat equation}. When the underlying manifold is compact, this was first studied indepedently  by X. Cao \cite{xCao2008} and Kuang-Zhang \cite{KuZh2008}.

Suppose that $u$ is a positive solution of \eqref{conjugate heat equation}. Let $u=(4\pi\tau)^{-\frac{n}{2}}e^{-v}$, $\tau=T-t$, Then we have
$$v=-\ln u-\frac{n}{2}\ln(4\pi\tau),$$
and the following evolution equations:
$$\frac{\partial g_{ij}}{\partial \tau}=2R_{ij},$$
$$\frac{\partial u}{\partial \tau}=\Delta u-Ru,$$
and
$$\frac{\partial v}{\partial \tau}=\Delta v-|\nabla v|^2+R-\frac{n}{2\tau}.$$

Define
\begin{align*}
F&=|\nabla v|^2-bv_{\tau}+cR+d\frac{v}{\tau}\\
&=|\nabla v|^2-b\Delta v+b|\nabla v|^2-bR+\frac{bn}{2\tau}+cR+d\frac{v}{\tau}\\
&=(b+1)|\nabla v|^2-b\Delta v+(c-b)R+d\frac{v}{\tau}+\frac{bn}{2\tau}.
\end{align*}

\begin{lemma}
For the operator $\mathcal{L}=\frac{\partial}{\partial \tau}-\Delta$, one has
\begin{align}\label{evolution conjugate F}
\mathcal{L}F&=-2\nabla_i F\nabla_i v-(2c-b)\Delta R+2(c+1)\nabla_i R\nabla_i v+\frac{d}{\tau}|\nabla f|^2-2(b+2)R_{ij}\nabla_iv\nabla_j v\nonumber\\
&\quad -2|\nabla^2 v|^2+2bR_{ij}\nabla_i\nabla_j v-2(c-b)|Rc|^2+\frac{d}{\tau}R-\frac{d}{\tau^2}v-\frac{(b+d)n}{2\tau^2}.
\end{align}
\end{lemma}

\begin{proof} It is straightforward to check that
\begin{equation}
\begin{aligned}
\frac{\partial |\nabla v|^2}{\partial \tau}&=-2R_{ij}\nabla_iv\nabla_jv+2\nabla_iv_{\tau}\nabla_i v\\
&=-2R_{ij}\nabla_iv\nabla_jv+2\nabla_i(\Delta v-|\nabla v|^2+R-\frac{n}{2\tau})\nabla_iv\\
&=\Delta|\nabla v|^2-4R_{ij}\nabla_iv\nabla_jv-2|\nabla^2 v|^2-2\nabla_i|\nabla v|^2\nabla_iv+2\nabla_i R\nabla_iv,
\end{aligned}
\end{equation}
\begin{equation}
\begin{aligned}
\frac{\partial \Delta v}{\partial \tau}&=-2R_{ij}\nabla_i\nabla_j v+\Delta v_{\tau}\\
&=\Delta^2v-2R_{ij}\nabla_i\nabla_j v-2\nabla_i\Delta v\nabla_i v-2|\nabla^2 v|^2-2R_{ij}\nabla_iv\nabla_j v +\Delta R,
\end{aligned}
\end{equation}
\begin{equation}\frac{\partial R}{\partial \tau}=-\Delta R-2|Rc|^2,
\end{equation}
and
\begin{equation}\frac{\partial}{\partial \tau}(\frac{v}{\tau})=\Delta (\frac{v}{\tau})-\frac{1}{\tau}|\nabla v|^2+\frac{R}{\tau}-\frac{n}{2\tau^2}-\frac{v}{\tau^2}.\end{equation}
From the formulas above, the lemma follows immediately.
\end{proof}

By setting $b=-2$, $c=-1$ and $d=0$, it follows from \eqref{evolution conjugate F} that
$$\mathcal{L}F=-2\nabla_iF\nabla_i v - 2|R_{ij}+\nabla_i\nabla_j v|^2+\frac{n}{\tau^2}.$$
Let $H=-|\nabla v|^2+2\Delta v+R$, then
\begin{equation}\label{evolution conjugate H}
\mathcal{L}H=-2\nabla_iH\nabla_i v - 2|R_{ij}+\nabla_i\nabla_j v|^2.
\end{equation}

Moreover, we may write $H=2z-y,$ where $y=|\nabla v|^2+R$ and $z=\Delta v+R$.

\begin{proposition}
Let $M^n$ be a complete manifold, and $g(t)$, $t\in[0, T]$, be a complete solution to the Ricci flow on $M$ with bounded Ricci curvature and nonnegative scalar curvature. If $u$ is a positive solution to \eqref{conjugate heat equation}, then for $v=-\ln u-\frac{n}{2}\ln(4\pi\tau)$, $\tau=T-t$, any point $O\in M$ and any constant $R_0>0$, we have
$$-|\nabla v|^2+2\Delta v+R\leq 2n[\frac{1}{\tau}+16 R_{max}+\frac{(72n+128)}{R_0^2}+\frac{8\sqrt{(n-1)R_{max}}}{R_0}]$$
on $\coprod_{\tau\in(0, T]} B_{\tau}(O, R_0)\times\{\tau\}$, where $R_{max}=\max_{[0,T]}|Rc|$.
\end{proposition}

\begin{proof}
Let
$\displaystyle \phi(x,\tau)=\eta\left(\frac{r(x,\tau)}{2R_0}\right)$
on $\coprod_{\tau\in[0, T]} B_{\tau}(O, 2R_0)\times\{\tau\}$, where $\eta(s)$ is the same function as in the proof of Proposition \ref{local harnack p<1 1<b<2}. From \eqref{evolution conjugate H}, we have
\begin{align*}
\tau\phi\mathcal{L}(\tau\phi H)
&=\tau\phi^2H+\tau^2\phi H\phi_\tau-\tau^2\phi H\Delta\phi - 2\tau^2\phi\nabla_i\phi\nabla_i H -2\tau^2\phi^2\nabla_iH\nabla_i v \\
&\quad - 2\tau^2\phi^2|R_{ij}+\nabla_i\nabla_j v|^2.
\end{align*}

Let $\tilde{H}=\tau\phi H=2\tilde{z}-\tilde{y}$, $\tilde{y}=\tau\phi y$, and $\tilde{z}=\tau\phi z$. Since
$$|\nabla \phi|\leq \frac{8}{R_0}\phi^{\frac{1}{2}}, \quad \Delta\phi \geq-\frac{8n}{R^2_0}-\frac{8\sqrt{(n-1)R_{max}}}{R_0},\quad \frac{\partial \phi}{\partial \tau}\leq 16R_{max},
$$
and
$$-2\tau_0^2\phi^2\nabla_iH\nabla_i v=2\tau^2_0\phi H\nabla_i\phi\nabla_i v\leq\frac{16}{R_0}\tau_0^{\frac{1}{2}}\tilde{y}^{\frac{1}{2}}(2\tilde{z}-\tilde{y}),$$
if $\tilde{H}$ reaches a positive maximum at $(x_0,\tau_0)$, we have
\begin{align*}
0&\leq \tau_0\phi\mathcal{L}(\tau\phi H)\\
&\leq [1+16\tau_0 R_{max}+\frac{(8n+128)\tau_0}{R_0^2}+\frac{8\tau_0\sqrt{(n-1)R_{max}}}{R_0}](2\tilde{z}-\tilde{y})\\
&\quad +\frac{16}{R_0}\tau_0^{\frac{1}{2}}\tilde{y}^{\frac{1}{2}}(2\tilde{z}-\tilde{y})-\frac{2}{n}|\tilde{z}|^2\\
&= [1+16\tau_0 R_{max}+\frac{(8n+128)\tau_0}{R_0^2}+\frac{8\tau_0\sqrt{(n-1)R_{max}}}{R_0}](2\tilde{z}-\tilde{y})\\
&\quad +\frac{16}{R_0}\tau_0^{\frac{1}{2}}\tilde{y}^{\frac{1}{2}}(2\tilde{z}-\tilde{y})-\frac{1}{2n}(2\tilde{z}-\tilde{y}+\tilde{y})^2\\
&\leq -\frac{1}{2n}(2\tilde{z}-\tilde{y})^2+[1+16\tau_0 R_{max}+\frac{(72n+128)\tau_0}{R_0^2}+\frac{8\tau_0\sqrt{(n-1)R_{max}}}{R_0}](2\tilde{z}-\tilde{y}).
\end{align*}

Therefore, at $(x_0,\tau_0)$, we have
$$\tilde{H}\leq 2n[1+16\tau_0 R_{max}+\frac{(72n+128)\tau_0}{R_0^2}+\frac{8\tau_0\sqrt{(n-1)R_{max}}}{R_0}].$$

Hence for any $(x,\tau)$, one has
$$\tau H\leq 2n[1+16\tau R_{max}+\frac{(72n+128)\tau}{R_0^2}+\frac{8\tau\sqrt{(n-1)R_{max}}}{R_0}].$$

\end{proof}

Letting $R_0\rightarrow \infty$ in the proposition above, we have

\begin{corollary}
Let $M^n$ be a complete manifold, and $g(t)$, $t\in[0, T]$, be a complete solution to the Ricci flow on $M$ with bounded Ricci curvature and nonnegative scalar curvature. If $u$ is a positive solution to \eqref{conjugate heat equation}, then for $v=-\ln u-\frac{n}{2}\ln(4\pi\tau)$, $\tau=T-t$, we have
$$-|\nabla v|^2+2\Delta v+R\leq 2n[\frac{1}{\tau}+16 R_{max}],$$
on $M\times[0,T)$, where $R_{max}=\max_{M\times[0,T]}|Rc|$.
\end{corollary}


Now we are ready to proveTheorem \ref{theorem conjugate heat equation}.
\begin{proof}[Proof of Theorem \ref{theorem conjugate heat equation}:]
Let $\hat{H}=\tau H-k$, then \eqref{evolution conjugate H} implies that
$$\tau \mathcal{L}\hat{H}=\hat{H}+k - 2\tau\nabla_i\hat{H}\nabla_i v - 2\tau^2|R_{ij}+\nabla_i\nabla_j v|^2.$$

Let $\check{H}=\hat{H}-\epsilon\psi$, where $\psi=e^{A\tau}f$ for some constant $A$ to be determined and $f$ is the same function as in Lemma \ref{distance function}, then
\begin{align*}
\tau\mathcal{L}(\check{H})&=\tau\mathcal{L}(\hat{H})-A\tau\epsilon\psi+\tau\epsilon e^{A\tau}\Delta f\\
&\leq \check{H}+k+\epsilon\psi - 2\tau\nabla_i\check{H}\nabla_i v-2\tau\epsilon e^{A\tau}\nabla_i f\nabla_i v-\frac{1}{2n}(\check{H}+k+\epsilon\psi+\tau y)^2-(A-C)\tau\epsilon\psi\\
&\leq \check{H}+k+\epsilon\psi- 2\tau\nabla_i\check{H}\nabla_i v+n\tau\epsilon e^{A\tau}|\nabla f|^2 + \frac{1}{n}\tau\epsilon e^{A\tau}|\nabla v|^2\\
&\quad  -\frac{1}{2n}(\check{H}+k+\epsilon\psi+\tau y)^2-(A-C)\tau\epsilon\psi.
\end{align*}

Assume that at time $\tau_0$ and some point $x_0$, $\tilde{H}$ reaches $0$ for the first time, then
\begin{align*}
0\leq\tau_0\mathcal{L}(\check{H})&\leq k+\epsilon\psi+n\tau_0\epsilon e^{A\tau_0}|\nabla f|^2 + \frac{1}{n}\tau_0\epsilon e^{A\tau_0}|\nabla v|^2-\frac{1}{2n}(k+\epsilon\psi+\tau_0 y)^2-(A-C)\tau_0\epsilon\psi\\
&\leq k-\frac{k^2}{2n}+\epsilon\psi-\frac{k}{n}\epsilon\psi+\frac{1}{n}\tau_0\epsilon\psi y - \frac{1}{n}(k+\epsilon\psi)\tau_0y-\frac{1}{2n}\tau_0^2y^2-[A-(n+1)C]\tau_0\epsilon\psi.
\end{align*}

Thus, if we pick $k=2n$ and $A>(n+1)C$, there is a contradiction. It means that
$$\hat{H}-\epsilon\psi<0$$
for any $\epsilon>0$.

Letting $\epsilon\rightarrow 0$ yields $\hat{H}\leq0$, i.e.,
$$2\Delta v-|\nabla v|^2+R\leq \frac{2n}{\tau}.$$
\end{proof}

\begin{remark}
Theorem \ref{theorem conjugate heat equation} generalizes the Li-Yau type estimate \eqref{kuang-zhang} obtained by X. Cao \cite{xCao2008} and Kuang-Zhang \cite{KuZh2008} to the complete noncompact setting.
\end{remark}

\bigskip

\section{Li-Yau-Hamilton estimates for the heat equation}

In this last section, we apply our arguments in Section 2 to the heat equation with potential under the Ricci flow and derive an Li-Yau type estimate for positive solutions on complete noncompact manifolds, extening the previous work of  X. Cao and R. Hamilton \cite{CaHa2009} in the compact case.
 More precisely, we consider the following heat equation with a linear forcing term,
\begin{equation}\label{heat equation}
\frac{\partial u}{\partial t}=\Delta u + Ru,
\end{equation}
coupled with the Ricci flow (1.5).

For any positive solution $u$ to \eqref{heat equation}, let  $v=-\ln u$. Then, it is easy to check that
$$\frac{\partial v}{\partial t}=\Delta v-|\nabla v|^2-R.$$

Let
\begin{align*}
F&=|\nabla v|^2 + bv_t+ cR\\
&=b\Delta v+(1-b)|\nabla v|^2+(c-b)R
\end{align*}

and
$$\mathcal{L}=\frac{\partial}{\partial t}-\Delta.$$

Similarly to Proposition 2.1 in \cite{CaZh2014}, one can show
\begin{proposition}
\begin{align*}
\mathcal{L}(F)&= -2\nabla_i F\nabla_i v - 2|\nabla^2 v|^2 +2bR_{ij}\nabla_i\nabla_j v+2c|Rc|^2\\
&\quad\ -b\frac{\partial R}{\partial t}+2(c-1)\nabla_i R\nabla_i v - 2bR_{ij}\nabla_i v\nabla_j v.
\end{align*}
\end{proposition}

By choosing appropriate $c$ and a cut-off function, it is not hard to drive the following local differential Harnack estimate.


\begin{proposition}\label{local harnack p=1}
Let $(M^n, g_{ij}(t))$, $t\in[0, T]$, be a complete solution to the Ricci flow with bounded curvature and nonnegative curvature operator. If $u$ is a positive solution to \eqref{heat equation}, then for $v=-\ln u$, $b>1$, any point $O\in M$ and any constant $R_0>0$, we have
\begin{align}\label{eq local harnack p=1 b<2}
&|\nabla v|^2+bv_t+(1-b)R\leq \frac{b^2n}{2}\left[\frac{1}{t}+16R_{max}+\frac{C_0}{R_0^2}+ C_1(R_{max},t)\right]
\end{align}
on $\coprod_{t\in(0, T]}B_t(O, R_0)\times\{t\}$ when $1<b\leq 2$, and
\begin{align}\label{eq local harnack p=1 b>2}
&|\nabla v|^2+bv_t+(1-b)R
\leq \frac{b^2n}{2}\left[\frac{1}{t}+(16+\frac{2(b-2)}{bn})R_{max}+\frac{C_0}{R_0^2}+C_1(R_{max},t)\right]
\end{align}
on $\coprod_{t\in(0, T]}B_t(O, R_0)\times\{t\}$ when $b\geq 2$. Here $C_0=8n+128+\frac{16b^2n}{b-1}$, $C_1(R_{max},t)=\sqrt{\frac{2}{b^2n}[\frac{(b-2)^2}{2}R_{max}^2+\frac{bR_{max}}{t}]}$, and $R_{max}=\sup_{M\times[0,T]}R$.
\end{proposition}

\begin{proof} By taking $c=1-b$, we have
$$F=|\nabla v|^2+bv_t+(1-b)R=y+bz,$$
where $y=|\nabla v|^2 + R$ and $z=v_t-R$. Hence, by Proposition 4.1,
\begin{align}\label{eq5}
\mathcal{L}(F)
&= -2\nabla_i F\nabla_i v - 2|\nabla^2 v-\frac{b}{2}Rc|^2+\frac{(b-2)^2}{2}|Rc|^2\\
&\quad\ -b[R_t + 2\nabla_i R\nabla_i v + 2R_{ij}\nabla_i v\nabla_j v].
\end{align}
Let $\phi(x,t)$ be the cut-off function in \eqref{eq phi}. Then, it follows that
\begin{align*}
t\phi\mathcal{L}(t\phi F)
&=\phi(\tilde{y}+b\tilde{z})+t\phi_t(\tilde{y}+b\tilde{z})-t\Delta\phi(\tilde{y}+b\tilde{z}) -2t^2\phi\nabla_i\phi\nabla_i F -2t^2\phi^2\nabla_i F\nabla_i v\\
&\quad\ - 2t^2\phi^2|\nabla^2 v-\frac{b}{2}Rc|^2+\frac{(b-2)^2}{2}t^2\phi^2|Rc|^2-bt\phi^2Q+bt\phi^2R,
\end{align*}
where $Q$ is the quantity in \eqref{eq Q}, $\tilde{y}=t\phi y$, and $\tilde{z}=t\phi z$.

Recall that
$$|\frac{\partial \phi}{\partial t}|\leq 16R_{max},\qquad \Delta\phi\geq-\frac{8n}{R^2_0}, \qquad |\nabla \phi|\leq \frac{8}{R_0}\phi^{\frac{1}{2}}.$$

If $t\phi F\leq 0$ in $\coprod_{t\in[0, T]}B_t(O, 2R_0)\times\{t\}$, then we are done. Otherwise, since $t\phi F=0$ on the parabolic boundary of $\coprod_{t\in[0, T]}B_t(O, 2R_0)\times\{t\}$, we may assume that $t\phi F$ achieves a positive maximum for the first time at some $t_0>0$ and some interior point $x_0$. Thus at $(x_0, t_0)$, we have
$$\tilde{y}+b\tilde{z}=t_0\phi F>0, \qquad
F\nabla\phi=-\phi\nabla F,\qquad \textrm{and}\qquad \mathcal{L}(t\phi F)(x_0,t_0)\geq0.$$

Since $$-2t_0^2\phi^2\nabla_i F\nabla_i v=2t_0^2\phi F\nabla_i\phi\nabla_iv\leq \frac{16}{R_0}t_0^{\frac{1}{2}}\tilde{y}^{\frac{1}{2}}(\tilde{y}+b\tilde{z}),$$
we have
\begin{align*}
0&\leq t_0\phi\mathcal{L}(t\phi F)\\
&\leq (\tilde{y}+b\tilde{z})+16t_0R_{max}(\tilde{y}+b\tilde{z})+\frac{(8n+128)t_0}{R_0^2}(\tilde{y}+b\tilde{z}) +\frac{16}{R_0}t_0^{\frac{1}{2}}\tilde{y}^{\frac{1}{2}}(\tilde{y}+b\tilde{z})\\
&\quad\ - \frac{2}{n}t_0^2\phi^2(y+z+\frac{2-b}{2}R)^2+\frac{(b-2)^2}{2}t_0^2\phi^2R^2+bt_0\phi^2R\\
&\leq (\tilde{y}+b\tilde{z}) \left[1+16t_0R_{max}-\frac{4(b-1)}{b^2n}\tilde{y}+\frac{16}{R_0}t_0^{\frac{1}{2}}\tilde{y}^{\frac{1}{2}}+\frac{(8n+128)t_0}{R_0^2}\right] -\frac{2(b-1)^2}{b^2n}\tilde{y}^2\\
&\quad\ - \frac{2}{b^2n}(\tilde{y}+b\tilde{z})^2 +\frac{(b-2)^2(n-1)}{2n}t_0^2\phi^2R^2 -\frac{2(b-1)(2-b)}{bn}\tilde{y}t_0\phi R\\
&\quad\ - \frac{2(2-b)}{bn}t_0\phi R(\tilde{y}+b\tilde{z})+bt_0\phi R.
\end{align*}

If $1<b\leq 2$, we have
\begin{align*}
0&\leq -\frac{2}{b^2n}(\tilde{y}+b\tilde{z})^2+\left[1+16t_0R_{max}-\frac{4(b-1)}{b^2n}\tilde{y}+\frac{16}{R_0}t_0^{\frac{1}{2}}\tilde{y}^{\frac{1}{2}}+\frac{(8n+128)t_0}{R_0^2}\right](\tilde{y}+b\tilde{z})\\
&\quad\ +\frac{(b-2)^2(n-1)}{2n}t_0^2R_{max}^2+bt_0R_{max}\\
&\leq -\frac{2}{b^2n}(\tilde{y}+b\tilde{z})^2+\left[1+16t_0R_{max}+(8n+128+\frac{16b^2n}{b-1})\frac{t_0}{R_0^2}\right](\tilde{y}+b\tilde{z})\\
&\quad\ +\frac{(b-2)^2(n-1)}{2n}t_0^2R_{max}^2+bt_0R_{max},
\end{align*}
which implies \eqref{eq local harnack p=1 b<2}.

If $b\geq 2$, then
\begin{align*}
0&\leq (\tilde{y}+b\tilde{z})+16t_0R_{max}(\tilde{y}+b\tilde{z})+\frac{(8n+128)t_0}{R_0^2}(\tilde{y}+b\tilde{z}) +\frac{16}{R_0}t_0^{\frac{1}{2}}\tilde{y}^{\frac{1}{2}}(\tilde{y}+b\tilde{z})-\frac{2(b-1)^2}{b^2n}\tilde{y}^2\\
&\quad\ - \frac{2}{b^2n}(\tilde{y}+b\tilde{z})^2 +\frac{(b-2)^2(n-1)}{2n}t_0^2\phi^2R^2 -\frac{4(b-1)}{b^2n}\tilde{y}(\tilde{y}+b\tilde{z})+\frac{2(b-1)(b-2)}{bn}\tilde{y}t_0\phi R\\
&\quad\ + \frac{2(b-2)}{bn}t_0\phi R(\tilde{y}+b\tilde{z})+bt_0\phi R\\
&\leq -\frac{2}{b^2n}(\tilde{y}+b\tilde{z})^2+\left[1+(16+\frac{2(b-2)}{bn})t_0R_{max}+(8n+128+\frac{16b^2n}{b-1})\frac{t_0}{R_0^2}\right](\tilde{y}+b\tilde{z})\\
&\quad\ +\frac{(b-2)^2}{2}t_0^2R_{max}^2+bt_0R_{max},
\end{align*}
which leads to \eqref{eq local harnack p=1 b>2}. This finishes the proof.
\end{proof}

Letting $R_0\rightarrow \infty$, we have
\begin{corollary}\label{noncompact harnack p=1}
Let $(M^n, g_{ij}(t))$, $t\in[0, T]$, be a complete solution to the Ricci flow with bounded curvature and nonnegative curvature operator. If $u$ is a positive solution to \eqref{heat equation}, then for $v=-\ln u$, we have
\begin{align*}
&|\nabla v|^2+bv_t+(1-b)R\leq \frac{b^2n}{2}\left[\frac{1}{t}+16R_{max}+C_1(R_{max},t)\right]
\end{align*}
on $M\times(0, T]$ when $1<b\leq 2$, and
\begin{align*}
&|\nabla v|^2+bv_t+(1-b)R
\leq \frac{b^2n}{2}\left[\frac{1}{t}+(16+\frac{2(b-2)}{bn})R_{max}+C_1(R_{max},t)\right]
\end{align*}
on $M\times(0,T]$ when $b\geq 2$. Here $R_{max}$ and $C_1(R_{max},t)$ are the same constants as in Proposition \ref{local harnack p=1}.
\end{corollary}

Again, the above result can be refined as

\begin{theorem}
Let $(M^n, g_{ij}(t))$, $t\in[0, T]$, be a complete solution to the Ricci flow with bounded curvature and nonnegative curvature operator. If $u$ is a positive solution to \eqref{heat equation}, then for $v=-\ln u$, we have
\begin{align*}
&|\nabla v|^2+bv_t+(1-b)R-\frac{b(2-b)\sqrt{n-1}}{2}R_{max}-\frac{b^2n}{2t}\leq 0
\end{align*}
on $M\times(0, T]$ when $1<b\leq 2$, and
\begin{align*}
&|\nabla v|^2+bv_t+(1-b)R-b(b-2)[1+\frac{\sqrt{n-1}}{2}]R_{max}-\frac{b^2n}{2t}
\leq0
\end{align*}
on $M\times(0,T]$ when $b\geq 2$. Here  $R_{max}=\sup_{M\times[0,T]}R$.
\end{theorem}

\begin{proof} Let $H=t(F-K)-d$ with $d$ and $K$ to be determined, from \eqref{eq5} one has
\begin{align*}
t\mathcal{L}(H)\leq &H+d-2t\nabla_iH\nabla_iv- \frac{2}{b^2n}\left[H+tK+d+(b-1)ty+\frac{b(2-b)}{2}tR\right]^2\\
\quad\ &+\frac{(b-2)^2}{2}t^2R^2+btR\\
= &H+d-2t\nabla_iH\nabla_iv-\frac{2}{b^2n}(H+tK+d)^2-\frac{2(b-1)^2}{b^2n}t^2y^2+\frac{(b-2)^2(n-1)}{2n}t^2R^2\\
\quad\ &-\frac{4(b-1)}{b^2n}ty(H+tK+d)-\frac{2(2-b)}{bn}tR(H+tK+d)-\frac{2(b-1)(2-b)}{bn}t^2Ry+btR.
\end{align*}

Let $\tilde{H}=H-\epsilon\psi$, where $\psi=e^{At}f$ with some constant $A$ to be determined and $f$ being the same function as in Lemma \ref{distance function}. From Corollary \ref{noncompact harnack p=1}, we know that $\tilde{H}<0$ at $t=0$ and outside a fixed compact subset of $M$ for $t\in(0,T]$. Suppose that $\tilde{H}=0$ for the first time at $t=t_0>0$ and some point $x_0\in M$. Then,  at $(x_0, t_0)$,  we have
$$H=\epsilon\psi>0,\qquad \nabla\tilde{H}=0,\qquad\textrm{and}\qquad \mathcal{L}(\tilde{H})\geq 0.$$

If $1<b\leq 2$ and choose $d=\frac{b^2n}{2}$, then
\begin{align*}
0&\leq t_0\mathcal{L}(\tilde{H})\\
&\leq  \epsilon\psi+d+2t_0\epsilon e^{At_0}\nabla_if\nabla_iv-\frac{2}{b^2n}(\epsilon\psi+t_0K+d)^2-\frac{2(b-1)^2}{b^2n}t_0^2y^2+\frac{(b-2)^2(n-1)}{2n}t_0^2R^2\\
&\quad\ -\frac{4(b-1)}{b^2n}t_0y(\epsilon\psi+t_0K+d)-\frac{2(2-b)}{bn}t_0R(\epsilon\psi+t_0K+d)-\frac{2(b-1)(2-b)}{bn}t^2_0Ry\\
&\quad\ +bt_0R-At_0\epsilon\psi+t_0\epsilon e^{At_0}\Delta f\\
&\leq -\epsilon\psi-\frac{2}{b^2n}(\epsilon\psi)^2+\delta t_0\epsilon\psi|\nabla v|^2 + \frac{C}{4\delta} t_0\epsilon\psi-\frac{2(b-1)^2}{b^2n}t_0^2|\nabla v|^4\\
&\quad\ + \frac{(b-2)^2(n-1)}{2n}t_0^2R_{max}^2-\frac{2}{b^2n}t_0^2K^2-(2-b)(b-1)t_0R-At_0\epsilon\psi+Ct_0\epsilon\psi.
\end{align*}
Now if we choose $K=\frac{b(2-b)\sqrt{n-1}}{2}R_{max}$, and some $\delta>0$ such that

$$-\frac{2}{b^2n}(\epsilon\psi)^2+\delta t_0\epsilon\psi|\nabla v|^2 -\frac{2(b-1)^2}{b^2n}t_0^2|\nabla v|^4\leq 0,$$
then we deduce
\begin{align*}
0\leq [(1+\frac{1}{4\delta})C-A]t_0\epsilon\psi.
\end{align*}
By further choosing $A>(1+\frac{1}{4\delta})C$, we get a contradiction.

If $b\geq2$ and $d=\frac{b^2n}{2}$, then
\begin{align*}
0&\leq t_0\mathcal{L}(\tilde{H})\\
&\leq  \epsilon\psi+d+2t_0\epsilon e^{At_0}\nabla_if\nabla_iv-\frac{2}{b^2n}(\epsilon\psi+t_0K+d)^2-\frac{2(b-1)^2}{b^2n}t_0^2y^2+\frac{(b-2)^2(n-1)}{2n}t_0^2R^2\\
&\quad\  -\frac{4(b-1)}{b^2n}t_0y(\epsilon\psi+t_0K+d)+\frac{2(b-2)}{bn}t_0R(\epsilon\psi+t_0K+d)+\frac{2(b-1)(b-2)}{bn}t^2_0Ry\\
&\quad\ +bt_0R-At_0\epsilon\psi+t_0\epsilon e^{At_0}\Delta f\\
&\leq -\epsilon\psi-(\frac{4}{b^2n}K-\frac{2(b-2)}{bn}R)t_0\epsilon\psi-\frac{2}{b^2n}(\epsilon\psi)^2+\delta t\epsilon\psi|\nabla v|^2+ \frac{C}{4\delta}t_0\epsilon\psi-\frac{2(b-1)^2}{b^2n}t_0^2|\nabla v|^4\\
&\quad\ +\frac{(b-2)^2(n-1)}{2n}t_0^2R^2+\frac{2(b-2)}{bn}t_0^2RK-\frac{2}{b^2n}t_0^2K^2+(b-2)(b-1)t_0R-2t_0K\\
&\quad\ -At_0\epsilon\psi+Ct_0\epsilon\psi-\frac{2(b-1)}{bn}[\frac{2}{b}K-(b-2)R]t_0^2y.
\end{align*}

Let $K=b(b-2)(1+\frac{\sqrt{n-1}}{2})R_{max}$, then
$$\frac{2}{b}K-(b-2)R\geq0,$$
$$(b-2)(b-1)R-2K\leq 0,$$
and
$$\frac{(b-2)^2(n-1)}{2n}R^2+\frac{2(b-2)}{bn}RK-\frac{2}{b^2n}K^2\leq0.$$
If we further choose some appropriate $\delta>0$ so that
$$-\frac{2}{b^2n}(\epsilon\psi)^2+\delta t\epsilon\psi|\nabla v|^2-\frac{2(b-1)^2}{b^2n}t_0^2|\nabla v|^4\leq0,$$
then
\begin{align*}
0\leq t_0\mathcal{L}(\tilde{H})\leq (-A+C+\frac{C}{4\delta})t_0\epsilon\psi.
\end{align*}
By choosing $A>C+\frac{C}{4\delta}$, we also get a contradiction.

Therefore, in either case, it must be true that $\tilde{H}<0$. Letting $\epsilon\rightarrow 0$, we conclude that $H\leq 0$. This finishes the proof of Theorem 3.4.
\end{proof}

In the special case when $b=2$, we obtain Theorem \ref{theorem heat equation}.

\begin{remark}
When $M$ is compact, Theorem \ref{theorem heat equation} was first proved by X. Cao and R. Hamilton in \cite{CaHa2009}.
\end{remark}

\bibliography{FDE}
\bibliographystyle{acm}

\end{document}